\documentclass[11 pt]{article}  

\usepackage[english]{babel}
\usepackage{mdframed,ulem}%,showkeys}
\newcommand{\ra}[1]{\overleftarrow{#1}}

\newcommand{\overbar}[1]{\overline{#1}}

\usepackage[utf8]{inputenc}
\usepackage{amsmath,indentfirst,qsymbols,graphicx,psfrag,amsfonts,stmaryrd,hyperref,color,enumerate,amsthm,ulem,dsfont}
\usepackage{url}
\usepackage[dvipsnames]{xcolor}
\usepackage{pgf,tikz}
\usetikzlibrary{arrows}
\usepackage{algorithm}
\usepackage{algorithmic}
\usepackage{listings}
\usepackage{eqparbox}
\usepackage{cleveref}
\usepackage{xparse}
\usepackage{subcaption}
\usepackage{todonotes}
\usepackage{comment}

\newtheorem{lem}{Lemma}
\newtheorem{defi}[lem]{Definition}
\newtheorem{pro}[lem]{Proposition}
\newtheorem{theo}[lem]{Theorem}
\newtheorem{cor}[lem]{Corollary}

\newtheorem{rem}[lem]{Remark}

\newcommand{\LET}{{\sf LastExitTree}}
\newcommand{\FET}{{\sf FirstEntranceTree}}

\def \bls{{\tiny $\blacksquare$ \normalsize }}
\def \sign{{\sf sign}}

\def \deg#1#2{{\sf Deg^{in}}_{#1}(#2)}
\def \degree{{\sf deg}}
\newcommand{\HC}{{\sf HeapCol}}
\newcommand{\OUTIN}{{\sf OUTIN}}

\DeclareMathOperator{\Id}{Id}

\def \1{\textbf{1}}

\def \Z{\mathbb{Z}}

\def \SP{\textsf{SpanningTrees}}

\def \HoleSet{{\sf HoleSet}}
\def \tb{{\tiny $\bullet$}}
\def \bpar#1{\left\{\begin{array}{#1} }
\def \epar { \end{array}\right.}
\newcommand{\Nb}{{\sf Nb}}
\newcommand{\Nbb}{{\sf Nb}}
\def \Paths{{\sf Paths}}
\def \app#1#2#3#4#5{\begin{array}{rccl} #1:&#2&\longrightarrow&#3\\ &#4&\longmapsto&#5\end{array}}
\def \Heap{{\sf Heap}}

\def \ba{\begin{align}}
\def \ea{\end{align}}
\def \be{\begin{eqnarray*}}
\def \ee{\end{eqnarray*}}
\def \ben{\begin{eqnarray}}
\def \een{\end{eqnarray}}

\def \beq{\begin{equation}}
\def \eq{\end{equation}}

\def \build#1#2#3{\mathrel{\mathop{\kern 0pt#1}\limits_{#2}^{#3}}}

\def \ba{{\bf a}}

\def \eref#1{(\ref{#1})}

\def \P{{\mathbb{P}}}

\def \imp{\Rightarrow}

\def \l{\left}

\def \r{\right}
\def \sous#1#2{\mathrel{\mathop{\kern 0pt#1}\limits_{#2}}}
\def \sur#1#2{\mathrel{\mathop{\kern 0pt#1}\limits^{#2}}}
\def \eqd{\sur{=}{(d)}}

\newqsymbol{`B}{\mathcal{B}}
\newqsymbol{`E}{\mathbb{E}}
\newqsymbol{`I}{\mathbb{I}}
\newqsymbol{`N}{\mathbb{N}}
\newqsymbol{`O}{\Omega}
\newqsymbol{`P}{\mathbb{P}}
\newqsymbol{`Q}{\mathbb{Q}}
\newqsymbol{`R}{\mathbb{R}}
\newqsymbol{`W}{\mathbb{W}}
\newqsymbol{`Z}{\mathbb{Z}}
\newqsymbol{`a}{\alpha}
\newqsymbol{`e}{\varepsilon}
\newqsymbol{`o}{\omega}
\newqsymbol{`t}{\tau}
\newqsymbol{`w}{{\cal W}}

\newcommand{\compact}{ \topsep0pt   \itemsep=0pt   \partopsep=0pt   \parsep=0pt}

\newcounter{c}
\def \bir{\begin{itemize}\compact \setcounter{c}{0}}
\def \itr{\addtocounter{c}{1}\item[($\roman{c}$)]} %i ii
\def \eir{\end{itemize}\vspace{-2em}~}

\newcounter{d}
\def \bia{\begin{itemize}\compact \setcounter{d}{0}}
\def \eia{\end{itemize}\vspace{-2em}~}

\newcounter{b}
\def \bi{\begin{itemize}\compact \setcounter{b}{0}}

\def \ei{\end{itemize}\vspace{-2em}~}

\tikzset{
	treenode/.style = {align=center, inner sep=0pt, text centered, font=\sffamily},
	arn_n/.style = {treenode, circle, white, draw=black,fill=black, text width=0.5em},
	arn_r/.style = {treenode, circle, white, draw=black, fill=red, text width=0.5em},
	arn_x/.style = {treenode, circle, white, draw=black, fill=white, text width=0.5em}
}

\begin{document}
\newcommand{\Torus}[1]{{\sf Torus}(#1)}
\newcommand{\cv}[1]{|#1|}
\newcommand{\W}{{\sf Weight}}
\renewcommand{\L}{{\sf Length}}

\def \ligne{\centerline{------------------------------}\\}

\newcounter{Mod}
\setcounter{Mod}{0}
\renewcommand{\theMod}{\Alph{Mod}}
\newcommand{\Modref}[1]{(\ref{#1})}

\newcommand \NewModel[3]{\refstepcounter{Mod}
\begin{mdframed}[topline=false,rightline=false,bottomline=false,linewidth=2pt]
\noindent {\bf Definition of Model \theMod :}  {\sf #1}.\label{#2}~\\#3
\end{mdframed}
}

\newcounter{SMod}
\setcounter{SMod}{0}
\renewcommand{\theSMod}{\Alph{SMod}}
\newcommand{\SModref}[1]{(\ref{#1})}

\newcommand\NewSModel[3]{\refstepcounter{SMod}
\begin{mdframed}[topline=false,rightline=false,bottomline=false,linewidth=2pt]
\noindent {\bf Definition of Subtree of tree Model \theSMod:}  {\sf #1}.\label{#2}~\\#3
\end{mdframed}}

\newcounter{Ker}
\setcounter{Ker}{0}
\renewcommand{\theKer}{\Alph{Ker}}
\newcommand{\Kerref}[1]{(\ref{#1})}

\newcommand \NewKernel[3]
{\refstepcounter{Ker}
\begin{mdframed}[topline=false,rightline=false,bottomline=false,linewidth=2pt]
\noindent{\bf Definition of the kernel $K^{(\theKer)}$ :} {\sf #1}\label{#2}.\\#3
\end{mdframed}}

\begin{center}
\huge\bf
A combinatorial proof of Aldous--Broder theorem for general Markov chains~\\
{\large \textbf{\bf Luis Fredes$^{\dagger}$ and Jean-Fran\c{c}ois Marckert$^{*}$}}
\rm \\
\large{$^\dagger$Universit\'e Paris-Saclay.\\
	$^{*}$CNRS, LaBRI, Universit\'e Bordeaux}
\normalsize 
\end{center}

\begin{abstract} Aldous--Broder algorithm is a famous algorithm used to sample a uniform spanning tree of any finite connected graph $G$, but it is more general: given an irreducible and reversible Markov chain $M$ on $G$ started at $r$, the tree rooted at $r$ formed by the first entrance steps in each node (different from the root) has a probability proportional to $\prod_{e=(e^{-},e^+)\in {\sf Edges}(t,r)} M_{e^{-},e^+}$, where the edges are directed toward $r$. In this paper we give proofs of Aldous--Broder theorem in the general case, where the kernel $M$ is irreducible but not assumed to be reversible (this generalized version appeared recently in  Hu, Lyons and Tang \cite{HLT}).\par 
We provide two new proofs: an adaptation of the classical argument, which is purely probabilistic, and a new proof based on combinatorial arguments. On the way we introduce a new combinatorial object that we call the golf sequences.
              \end{abstract}

\subsection*{Acknowledgments} The first author acknowledge support from ERC 740943 \emph{GeoBrown}.

\section{Introduction}

\label{sec:intro} Consider $G=(V,E)$ a finite connected graph: $V$ is the set of vertices, and $E$ the set of undirected edges. A path is a sequence of vertices  $w=(w_k, 0\leq k \leq m)$, with the property  that $\{w_k,w_{k+1}\}\in E$ for every $k$.
This path is said to be covering if it visits all the vertices and if $m$ is the first time with this property. More formally, for every $1\leq k \leq |V|$ define
\be
\tau_k(w) = \inf\{j, |\{w_0,\cdots,w_j\}|=k\},~~~1\leq k \leq |V|,\ee
the first time the path has visited $k$ different points (we write $\tau_k$ instead of $\tau_k(w)$ when it is clear from the context). The path $w$ is then called covering if $\tau_{|V|}(w)=m$. 
\par
Denote by $\SP(G)$ the set of spanning trees $t$ of $G$, and by $\SP^{\bullet}(G)$ the set of rooted spanning trees $(t,r)$, where $r$, the root, is a distinguished vertex of $V$. Each tree $(t,r)$ is considered as a directed graph in which the edges are directed toward the root: for any $(e^{-},e^+)$ in the edge set $E(t,r)$, $e^+$ is the parent of $e^{-}$ in $t$ (we write $e^+=p(e^{-})$).
\begin{defi}\label{defi:dqsd} For a covering path $w$, denote by $\FET(w)$ the rooted spanning tree $(t,w_0)$ whose $|V|-1$ edges are $(w_{\tau_k},w_{\tau_k-1})$ (that is oriented towards the root $w_0$) for $2\leq k \leq \tau_{|V|}$.
\end{defi}
\noindent Thus, each edge corresponds to the first edge used to visit each vertex different from $w_0$, turned around.
\par
A Markov kernel $M=(M_{a,b},a,b \in V)$ on $V$ is said to be {\it positive on } $G$ if $M_{a,b}>0 \iff \{a,b\}\in E$.
Any Markov chain driven by such a Markov kernel can traverse each edge of $G$ both ways, and each step of such a chain corresponds to an edge of $G$. On a connected graph, as it is the case here, this chain is irreducible, therefore it has a unique invariant distribution, $\rho=(\rho_v,v\in V)$.\par
A model of random covering paths is obtained by killing a Markov chain $W$ at the cover time.
\begin{theo}\label{theo:AB}[Aldous--Broder, \cite{Al90} and \cite{Bro89}]Let $W$ be a Markov chain with positive and \textbf{reversible} kernel $M$ and invariant distribution $\rho$. Then, for any $(t,r)\in \SP^{\bullet}(G)$
  \ben\label{eq:qsdq17}
  \P\l[\FET\l(W_0,\cdots,W_{\tau_{|V|}}\r)=(t,r)~|~W_0=r\r]= {\sf Const.}\l[\prod_{e\in E(t,r)} M_e\r]/\rho(r).\een
\end{theo}
To be totally clear, each $e$ is an edge of the type $(e^{-},e^+)$, where $e^+=p(e^{-})$, and $M_e=M_{e^{-},e^+}$.\medbreak
Elements around Aldous--Broder algorithm can be found in \cite{ Jarai09}.\medbreak

As a consequence if $M$ is the Markov kernel corresponding to the simple random walk on $G$, $M_{a,b}=1_{\{a,b\}\in E}/\degree_G(a)$, and since the invariant distribution of this simple random walk $\rho_v$ is proportional to $\degree_G(v)$, the r.h.s.\ of \eref{eq:qsdq17} is proportional to $\prod_{u\in V}1/\degree_G(u)$, that is independent of $t$, so that $\FET(W_0,\cdots,W_{\tau_{|V|}})$ is a uniform spanning tree\footnote{To conclude here, another argument is needed: the fact that the support of  $\FET(W_0,\cdots,W_{\tau_{|V|}})$ is the set of all spanning trees, which is easy to see}, rooted at $W_0=r$.

As we will see, Aldous--Broder theorem \textbf{does not hold} if we remove the reversibility condition. \medbreak 

Recall that if $X=(X_i,i\in \Z)$ is a Markov chain with positive kernel $M$ with invariant distribution $\rho$, under its stationary regime, then the time reversal of $X$ is also a Markov chain under its stationary regime, with same invariant distribution, and   with kernel
 \ben\label{eq:dqsd22}
 \ra{M}_{x,y}:= \rho_{y} M_{y,x} / \rho_x, \textrm{ for all }(x,y)\in V^2.
 \een
The next result extends Aldous--Broder theorem:
 \begin{theo}\label{theo:kgyqsd} Let $W$ be a Markov chain with positive kernel $M$ (\textbf{reversible or not}) with invariant distribution $\rho$.  Then, for any $(t,r)\in \SP^{\bullet}(G)$
\ben\label{eq:qsdq33}
\P\l[\FET\l(W_0,\cdots,W_{\tau_{|V|}}\r)=(t,r)~|~W_0=r\r]= {\sf Const.}\left[\prod_{e\in E(t,r)} \ra{M_e}\right]/\rho(r).
\een
\end{theo}
The aim of this paper, is to give two new proofs of this result. 
This theorem has been proved by Hu, Lyons and Tang \cite{HLT} shortly before our work; but, at the moment we were writing the present paper, this result was not present on the first version of \cite{HLT}, on the arxiv. 

\noindent $\bullet$ This theorem implies Aldous--Broder result since in the reversible case, $\ra{M}=M$.\\
$\bullet$ We will prove that if $W$ is taken \textbf{under its stationary regime} (meaning that $W_0\sim \rho$) then
\ben\label{eq:qsdq3b}
\P\l[\FET\l(W_0,\cdots,W_{\tau_{|V|}}\r)=(t,r)\r]= {\sf Const.}\left[\prod_{e\in E(t,r)} \ra{M_e}\right].
\een
This formula is equivalent to \eref{eq:qsdq33}. \\
\noindent $\bullet$ From our point of view, Aldous--Broder theorem should be stated like Theorem \ref{theo:kgyqsd}, since this time reversal is not only a tool for the (original) proof of the Theorem but a characteristic of its conclusion. As a matter of fact, it is probably common for researchers in the field to try to find a coupling between Wilson construction \cite{Wil96,PW98} of the spanning tree (which is faster) and Aldous--Broder's, but the fact that Wilson algorithm produces a tree with distribution ${\sf Const.}\prod_{e\in E(t,r)} M_e$ shows that it is likely not possible, unless ones use Wilson algorithm with $\ra{M}$ instead of $M$.\par
\noindent $\bullet$ Both theorems above are valid on multigraph, which are the analogous of graph in which multiple edges between vertices are allowed,  as well as loops (edges $\{a,a\}$ adjacent to a single node are allowed). All the proofs we give are valid in this settings too (even if tiny adjustments to treat multiple edges could be needed at some places).

\begin{rem}\label{rem:egfq} In general,  $\prod_{e\in E(t,r)} M_e$ and  $\prod_{e\in E(t,r)} \ra{M}_e$ are different. Let us write $\prod_{u\in t} $ for a product over the set of vertices of $t$.
  For $(t,r)\in\SP^{\bullet}(G)$,
  \[\prod_{e\in E(t,r)} M_e=\prod_{u\in t \setminus \{r\}} M_{u,p(u)}={\sf Const}.\;\rho_r\prod_{u\in t \setminus  \{r\}} \rho_u M_{u,p(u)}\]
  while by \eref{eq:dqsd22}, and since in a tree all nodes but the root are children of some nodes,
\be
\prod_{e\in E(t,r)} \ra{M_e}&=& \prod_{u\in t \setminus  \{r\}} \l[ M_{p(u),u}\,\rho_{p(u)}/\rho_u\r]
= {\sf Const.} \; \rho_r \prod_{u\in t \setminus  \{r\}} \rho_{p(u)}M_{p(u),u}. 
\ee
so that these quantities are different when $\rho$ is not reversible with respect to $M$.\par
Different formula does not imply different values, but a graph of size 3 is sufficient to observe that these formulas are indeed different: consider
\[M=\begin{bmatrix} 0 & 1/3 & 2/3\\ 1/5 & 0 & 4/5 \\ 1/7 & 6/7 & 0\end{bmatrix},\quad \rho=\frac{1}{226}\begin{bmatrix} 33,95,98\end{bmatrix}\quad \text{and} \quad \ra{M}=\begin{bmatrix} 0 & 19/33 & 14/33\\ 11/95 & 0 & 84/95 \\ 11/49 & 38/49 & 0\end{bmatrix}.\]
It can be verified that $\rho$ is the invariant measure of $M$ and that $\ra{M}$ is the reversal kernel.
The tree $(t,r)$ defined by $(2,1)$ and $(3,1)$ has weights
$M_{2,1}M_{3,1} = 1/35 \neq \ra{M}_{2,1}\ra{M}_{3,1}=121/4655$.
\end{rem}
Let us say some words on the paths $(w_k,0\leq k \leq \tau_{|V|})$ that satisfy $\FET\l(w_0,\cdots,w_{\tau_{|V|}}\r)=(t,r)$. 
Such a path starts at $w_0=r$ and eventually reaches each vertex $u$ of $V\setminus\{r\}$ for the first time from $p(u)$ its parent in $(t,r)$. In other words, $\FET\l(w_0,\cdots,w_{\tau_{|V|}}\r)=(t,r)$ iff
\ben\label{eq:fet} \{(w_{\tau_k},w_{\tau_k-1}), 2 \leq k \leq |V|\}= \{(u,p(u)), u \in V\setminus \{r\}\}, \een
meaning that $k\mapsto w(\tau_k)$ induces a decreasing labelling towards the root. \par 
The analysis of $\P(\FET\l(W_0,\cdots,W_{\tau_{|V|}}\r)=(t,r))$ from this kind of considerations leads to sum on all the paths according to the order at which the vertices are reached.
Since this direct approach of Theorem \ref{theo:kgyqsd} seems difficult to complete, the initial proofs by both Aldous and Broder of their theorems follow a tricky path which is purely probabilistic. In our new proof, we avoid also the direct treatment of the labellings discussed above, but produce instead a combinatorial argument, relying on a new object we introduce, the golf sequence, which, we believe is interesting in its own.

\subsection{Some recalls from combinatorics} 
\label{sec:srfc}

The new proof uses many ingredients of the combinatorics folklore, and  we must say that we have been deeply inspired by Zeilberger short paper \cite{DZ} where many results from linear algebra are proved by means of combinatorial tools, notably, Foata's proof of MacMahon's master theorem \cite{DZ, CF} and the famous
\begin{theo}[Matrix Tree Theorem]
	\ben\label{eq:qsftjhg}
	\det(\Id-M^{(r)})= \sum_{(t,r)\in\SP^{\bullet}(G)}\prod_{e\in E(t,r)} M_e,\een
	where $M^{(r)}$ is the matrix $M$ deprived of the line $r$ and column $r$, and $\Id$ the identity matrix with the same size as that of $M^{(r)}$. 
\end{theo}
\noindent
We borrow the argument from Zeilberger \cite{DZ} since it helps to present the arguments of our own proof.
\begin{proof}
 First,  a cycle is a (class of equivalence of) path $(w_0,\cdots,w_{m})$ for some $m\geq 1$, such that $w_m=w_0$ and $|\{w_0,\cdots,w_m\}|=m$ (it is simple). Two cycles are equivalent if one can be obtained from the other by shifting its indices in $\Z/m\Z$.

Observe that
\ben\label{eq:dsdQ}
\Id-M^{(r)} = \begin{bmatrix}\l(\sum_{v' \in V}M_{u,v'}\r)\1_{(u=v)} -M_{u,v} \end{bmatrix}_{(u,v)\in (V\setminus\{r\})^2}.
\een
Denote by ${\cal B}$ the set of pairs $(B,C)$ such that:\\
$\bullet$ $B$ is a directed graph on $V$, where each vertex of $V\setminus\{r\}$ has either 0 or 1 outgoing edge ending in $V$ (including $\{r\}$ this time). Denote by $V_B$ the set of vertices from which there is an outgoing edge.\\ 
$\bullet$ $C$ is a collection of directed disjoint cycles on $V_C=(V\setminus \{r\})\setminus V_B$.
\\
\indent Set $\W(B):=\prod_{u\in V_B} M_{u,t(u)}$ where $t(u)$ is the target of the edge starting at $u$; $\W(C):=(-1)^{N(C)}\prod_{c \textrm{ cycles of C}}\prod_{e \in c} M_{e^{-},e^+}$ the product of the weight of edges along the directed cycles of $C$, where $N(C)$ the number of cycles of $C$ and finally define $\W(B,C):=\W(B)\W(C)$. \\
\indent An expansion of \eref{eq:dsdQ} allows one to see that $\W({\cal B}):=\sum_{(B,C)\in {\cal B}} \W(B,C)=\det(\Id-M^{(r)})$. Here is the argument: first, expand $\det(\Id-M^{(r)})$ using Leibniz formula:
  \[\det(\Id-M^{(r)})= \sum (-1)^{\sign(\sigma)} \prod_{i\neq r} (\Id-M^{(r)})_{i,\sigma(i)},\] where the sum range on all permutations $\sigma$ on $V\setminus\{r\}$.
  Now, consider the set $F(\sigma)=\{i:\sigma(i)=i\}$ of fix points of $\sigma$, and rewrite, by \eref{eq:dsdQ}:
  \[\prod_{i \neq r}(\Id-M^{(r)})_{i,\sigma(i)}= \l(\prod_{i\in F(\sigma)} \l(-M_{i,i}+\sum_{j\in V}M_{i,j} \r)\r)\l(\prod_{i\in V\setminus (\{r\}\cup F(\sigma))} -M_{i,\sigma(i)}\r).\]
  The second parenthesis can be interpreted as the weight of cycles of $\sigma$ with lengths at least 2. Now expand the first parenthesis (without simplifying $M_{i,i}$, so that $M_{i,i}$ comes with a sign $+$, and in another term, with a sign $-$).
This first parenthesis can be rewritten as a sum over  $A\subset F(\sigma)$ as follows:
\[\prod_{i\in F(\sigma)} \l(-M_{i,i}+\sum_{j\in V}M_{i,j} \r)=\sum_{A \subset F(\sigma)} \l(\prod_{i \in A} (-M_{i,i}) \r) \l(\prod_{i \in F(\sigma)\setminus A} \sum_{j}M_{i,j}\r).\] 
Each factor $-M_{i,i}$ can be seen to be the weight of a loop over $i$ (that is a cycle of size $1$), and by expansion, $\prod_{i \in F(\sigma)\setminus A} \sum_{j\in V}M_{i,j}$ can be interpreted as the weight of a directed graph where each vertex of $F(\sigma)\setminus A$ has a single outgoing edge, ending on any vertex of $V$. This ends the argument explaining why $\W({\cal B})=\det(\Id-M^{(r)})$.
~\\
-- We claim now that $\W({\cal B})=\sum_{(t,r)\in\SP^{\bullet}(G)}\prod_{e\in E(t,r)} M_e$. The graphs ``$B$'' are made of cycles and trees, and $C$ is made of cycles (with a sign of the weight corresponding to parity). For any pair $(B,C)$ having (totally) at least one cycle one can define $(B',C')$ as follows: for a total order on the set of oriented cycles, take the greatest cycle $c$ in the union of $B$ and $C$. 
        Denote by $(B',C')$ the pair obtained by moving $c$ from the component containing it to the other.
        This map $(B,C) \to (B',C')$ is clearly an involution and satisfies $\W(B',C')=-\W(B,C)$. Hence, $\W({\cal B})$ coincides with the sum of the $\W(B,C)$ taken on the set of pairs $(B,C)$ which have no cycles: $C$ is empty, and the graph $B$ has no cycle, and since its number of edges is one less than its number of vertices, it is a spanning tree.
\end{proof}
\color{black}

\paragraph{The quantity $\det\l( \Id-M^{(r)}\r)$ is the central algebraic object;}
it is somehow the partition function of the family of weighted spanning trees under inspection, but not only. It will appear under the following forms at several places of the paper.

Notice that the direct expansion of the determinant (using permutations)   gives
  \ben\label{eq:fqs}\det( \Id-M^{(r)})=\sum_{C} (-1)^{N(C)}\prod_{c \textrm{ cycles of C}}\prod_{e \in c} M_{e}\een where the sum ranges over all sets $C$ of disjoint (oriented) cycles of length $\geq 1$   on the set of vertices $V\setminus \{r\}$.  \\
Now by considering the involution that changes the orientation of every cycle one gets
  \be\det( \Id-M^{(r)})=\det( \Id-\ra{M}^{(r)})\ee and {therefore for a fixed $r\in V$
\be \sum_{(t,r)\in\SP^{\bullet}(G)}\prod_{e\in E(t,r)} M_e= \sum_{(t,r)\in\SP^{\bullet}(G)}\prod_{e\in E(t,r)} \ra{M}_e.
\ee
This together with Remark \ref{rem:egfq} say that even though the terms associated to each tree $(t,r)$ may be different, these sums are equal (similar identities are present in Hu \& al. \cite{HLT}).
\par
Even more, the invariant distribution $\rho$ of the Markov kernels $M$ and $\ra{M}$ is related to these quantities by:
\ben\label{eq:qsdqms}
\rho_w={\sf Const}.\det(\Id-M^{(w)}) ={\sf Const}.\det(\Id-\ra{M}^{(w)})
\een 
so that, from the matrix tree theorem, this provides a connection between $\rho_w$ and the total mass of spanning trees rooted at $w$. To prove \eref{eq:qsdqms}, there are several methods, one of them being the use of Theorem \ref{theo:AB} or \ref{theo:kgyqsd}, but direct arguments exist (and are classical). One wants to solve the vector system $\l\{\rho(\Id-M)=0, \rho\l[\begin{smallmatrix} 1 \\\vdots \\1\end{smallmatrix}\r]=1\r\}$. Assume here that the vertex are labeled from $1$ to $n$. Since the sum by rows of $\Id-M$ is zero, the system is over-determined, so we discard the equation $(\rho(\Id-M))_{1,n}=0$. Then, the row vector $\rho$ is solution of
\ben\label{eq:gdokqs}
\rho Q =\begin{bmatrix} 0 & \cdots & 0 & 1\end{bmatrix}          \textrm{ where }     Q= \left[\begin{array}{c|c}  \begin{bmatrix}(\Id-M)_{i,j}\end{bmatrix}_{i \leq n, j \leq n-1} & \begin{array}{c} 1 \\\vdots \\1\end{array} \end{array}\right].                
\een
By symmetry, it suffices to prove that $\rho_n$ satisfies \eref{eq:qsdqms}. For this use the Cramer rule to get \eref{eq:qsdqms}: notice that replacing the $n$-th row of $Q$ by $[0,\cdots,0,1]$ we obtain the wanted determinant in the numerator (as a cofactor).  It remains to show for the denominator that
$\det( Q)
  =\sum_{w}\det(\Id-M^{(w)})$;
the following argument is classical. It suffices to expand $\det(Q)$ according to the last column and show that 
\ben\label{eq:fi} \det((\Id-M)_{(i,j): i \neq a, j \neq n})(-1)^{n-a}=\det((\Id-M)_{(i,j):i \neq a, j\neq a})=\det(\Id-M^{(a)}).\een
This first equality is more general: in fact, $\det((\Id-M)_{(i,j): i \neq a, j \neq b})(-1)^{b-a}= \det((\Id-M)_{(i,j):i \neq a, j\neq a})$ for all $b\in \{1,\cdots,n\}$, and this independence to the deprived column $b$ (up to the sign) is a consequence of the fact that the sum of the columns of $\Id-M$ is 0 (this property allows to play with the lacking column $b$: take the $a$-th column $C_a$ and replace it by $C_a+\sum_{k\neq a,b} C_k=\sum_{k\neq b} C_k=-C_b$. This replacement does not alter the determinant, and now, we have $-C_b$ at position $a$).

\section{New combinatorial proof of Theorem \ref{theo:kgyqsd}}

We will assume all along the proof that $|V|\geq 2$, otherwise the statement is trivial.\\
Theorem \ref{theo:kgyqsd} expresses $\P(\FET(W_i, 0\leq i \leq \tau_{|V|}(W))=(t,r))$ in terms of $\ra{M}$, but since $W$ is a Markov chain with kernel $M$, we need also to consider reverse paths in the proof.
For a path $w=(w_0,\cdots,w_m)$ on $G$ denote by $|w|=m$ its number of steps, $\overline{w}=w_m$ its last position, and $\underline{w}=w_0$ its first position. Denote by $\ra{w}=(w_m,w_{m-1},\cdots,w_0)$ the reverse time path and set }
\ben\label{eq:weightpath}
\W(w):=\prod_{i=0}^{|w|-1} \ra{M}_{w_{i+1},w_{i}},\een
which is the weight of $\ra{w}$ starting at $\overbar{w}$ taken under the kernel $\ra{M}$.

A \textit{heap $h$ of directed outgoing edges of a vertex $u$} is a finite sequence
\[h= [ (u,t_i), 1\leq i \leq \ell]\] for some $\ell\geq 0$, where, for each $i$, $\{u,t_i\}\in E$. The set of such heaps is $\cup_{k\geq 0} O_u^k$  where  $O_u=\{(u,v)~: \{u,v\}\in E\}$ is the set of oriented outgoing edges out of $u$ (in $G$).
A \textit{collection of heaps} is a sequence $H:=(H_u, u \in V)$ of heaps, where $H_u\in\cup_{k\geq 0} O_u^k$ is a heap of directed edges out of $u$ (see Fig. \ref{heaps}, picture 2). The set of collections of heaps is the product set
\[\HC=\prod_{u \in V} \Big(\bigcup_k O_u^k\Big).\]
Each element of $\HC$ can be seen as a multigraph (repetition of edges permitted), in which the sequence of outgoing edges out of each node is equipped with a total order.
We will use collections of heaps to encode various sorts of family of connected paths. As a combinatorial tool, they allow some rearrangements of steps, and then provide some ways to construct bijections efficiently.

The concatenation operation $\oplus$ of heaps is defined as usual:
\[[e_1,\cdots,e_m] \oplus [e_1', \cdots,e_{m'}']= [ e_1,\cdots,e_m,e'_1,\cdots,e'_{m'}]\] which extends to collection of heaps as expected :  $H\oplus H':=( H_u\oplus H'_  u, u \in V)$.\\
We say that $h$ is a prefix of $h\oplus h'$, notion which extends again to collections.

The weight of a heap  $h=[ (s_i,t_i),1\leq i \leq m]$ and of a collection $H=(H_u,u\in V)$ are 
\be
\W(h)&:=&\prod_{i=1}^m \ra{M}_{s_i,t_i},\\
\W(H) &:=&\prod_{u \in V} \W(H_u).
\ee
The weight of an empty heap is set to 1.

\paragraph{Prefix removal (PR).}
For  $h=h'\oplus h''$, the ``prefix removal'' of $h'$ in $h$ is possible. We write this operation
$h \ominus_{PR} h' = h''$. 
This extends to collections: if $H=H'\oplus H''$, we set $H \ominus_{PR} H' = H''$.

\paragraph{Trees.} 
The number of children $\deg{(t,r)}{u}$ of a node $u$ in $(t,r)$ which is often called out-degree or simply degree in the literature in the case of trees, coincides with the number of incoming edges at $u$ in $(t,r)$ (the notation  $\deg{(t,r)}{u}$ is introduced to avoid any confusion with $\degree_G(u)$, the degree  of the graph at $u$ which comes into play too, and with the total degree of $u$ in $t$). A node $u$ with $\deg{(t,r)}{u}=0$ is called a leaf, and the set of leaves is denoted $\partial t$. Notice that since $|V|\geq 2$, $r$ is not a leaf. The elements in the set $t^o=V \setminus \partial t$ are called internal nodes.

\subsection{Step 1. Collection of heaps encoding of a path}

\begin{defi}
  The passport of a collection of heaps $H=(H_u,u\in V)\in \HC$ is the pair
  \be
  \OUTIN(H):=\l[({\sf Out}_u(H),u \in V),({\sf In}_u(H),u \in V)\r]
  \ee 
  where, for all $u\in V$, ${\sf Out}_u(H):=|H_u|$ records the number of outgoing edges of $u$ and 
  \be
        {\sf In}_u(H)&:=& \sum_{w\in V} | (w,u) \textrm{ in } H_w|
        \ee
                          records the number of incoming edges to $u$ (with multiplicities).
\end{defi}
As a standard passport in real life, the passport is used to record the number of entries/exits at each node, with multiplicities. Different collections of heaps can have the same passport.  

Denote by $\Paths(r)$ the set of paths $w$ starting at $w_0=r$ (with any finite length). 
With any path $w\in \Paths(r)$ associate the collection of heaps $\Heap(w):=(\Heap_u(w), u \in V)\in \HC$, where
\[\Heap_u(w):=\l[ (w_{j+1},w_{j})~: 0 \leq j \leq |w|-1, w_{j+1}=u\r];\]
these edges are then taken according to their order in $w$, and they correspond to edges that reach $u$, that are turned around. It can be useful to consider them as steps outgoing from $u$ in $\ra{w}$, still taken according to their order in $w$ (see \Cref{heaps}).
\begin{figure}[h!]
	\centering
	\includegraphics[scale=0.5]{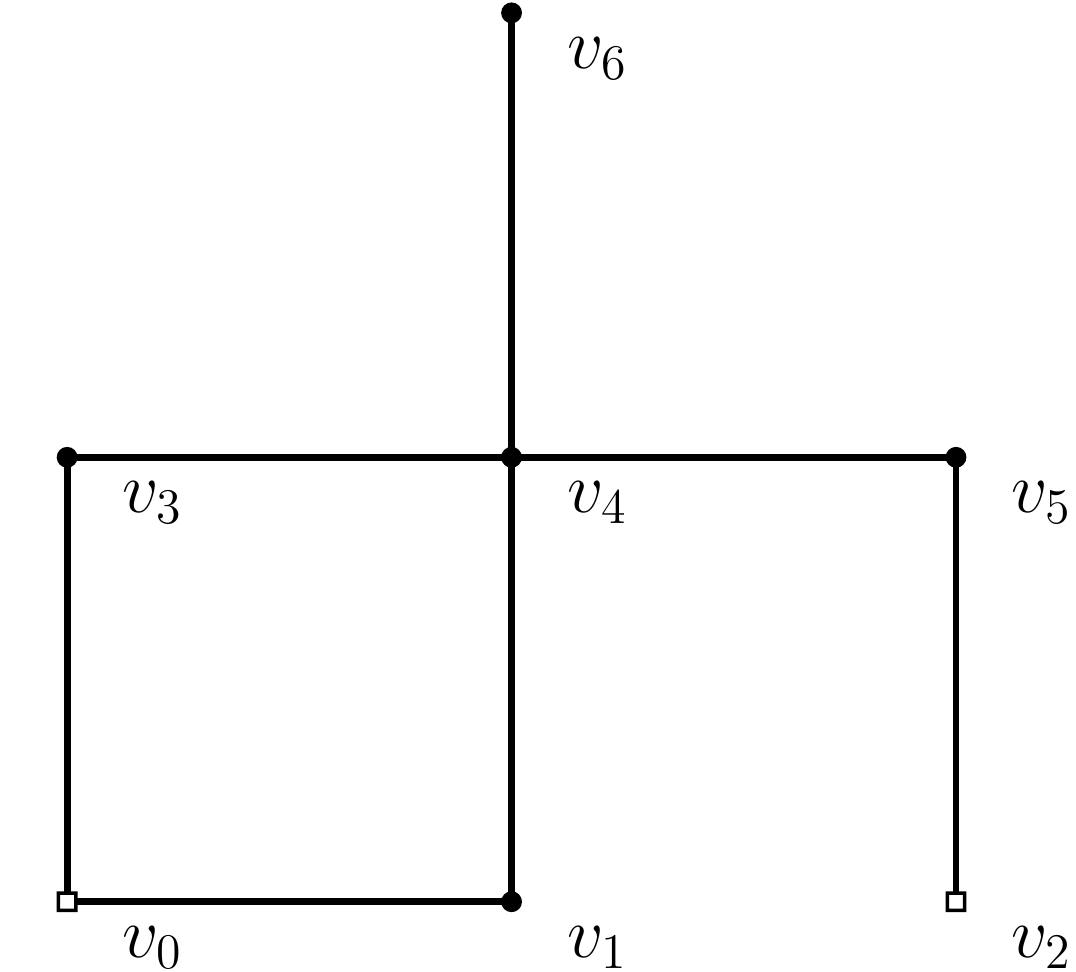}\hspace{1cm}\qquad \includegraphics[scale=0.5]{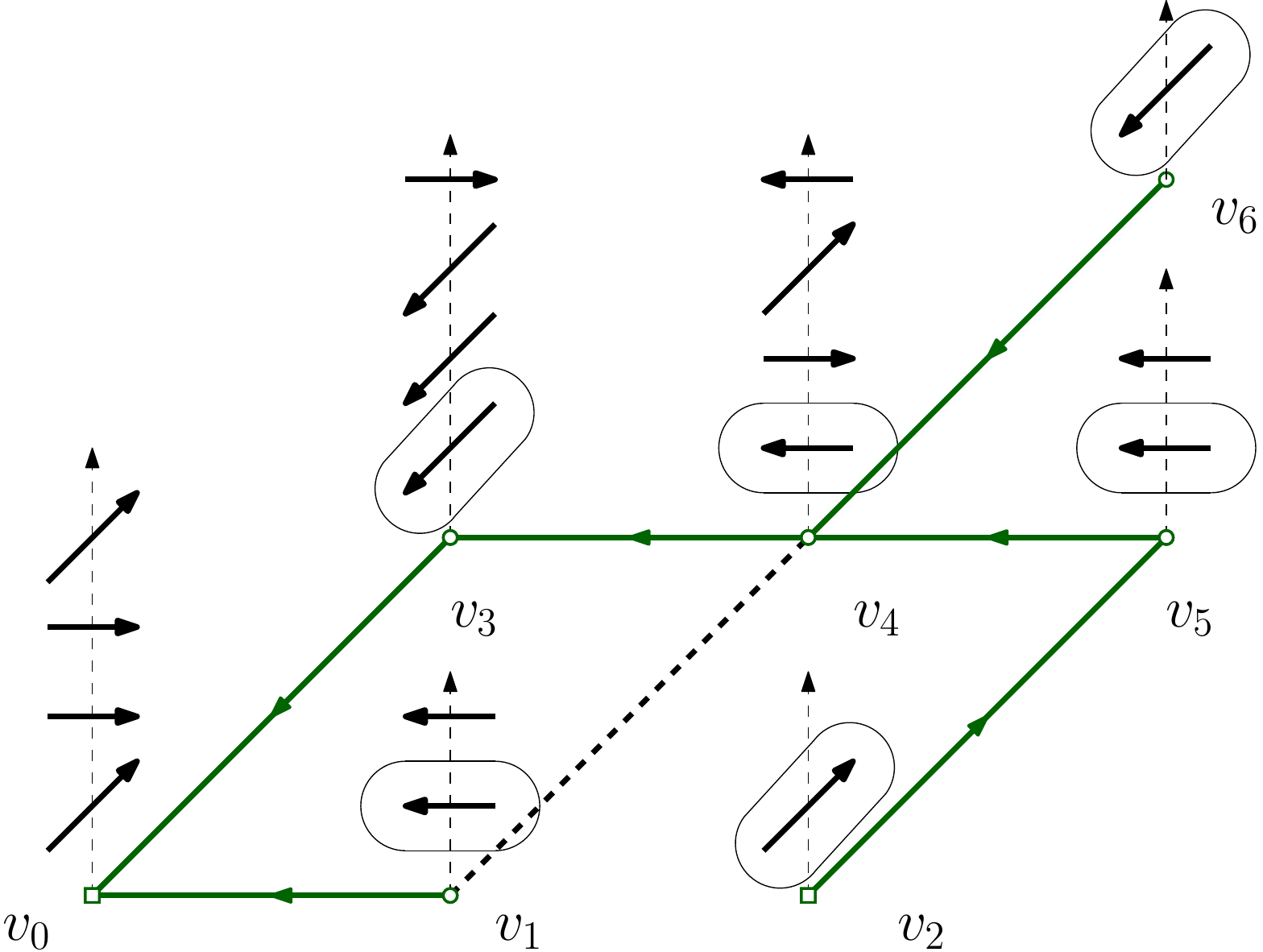}
	\caption{\label{heaps} Left: graph $G$, where the walk $w=v_0 v_3 v_0 v_1 v_0 v_1 v_0 v_3 v_0 v_3 v_4 v_5 v_4 v_6 v_4 v_3 v_4 v_5 v_2$ is performed. Right: Heap associated to the path $w$; for example, $H_{v_3}=[(v_3,v_0),(v_3,v_0),(v_3,v_0),(v_3,v_4)]$. To read the steps of $w$ in the heap, follow the edges from top to bottom to form $\protect\ra{w}$, then transform $\protect\ra{w}$ into $w$. The steps used to construct the tree are circled.}
	
\end{figure}

As warm up, let us prove the following classical fact.
\begin{lem}\label{lem:ujkgkuj}
  The map
  \[\app{\Heap}{{\Paths}(r)}{ \HC}{w}{\Heap(w)}\]
is a weight preserving injection (meaning that $\W(w)=\W(\Heap(w))$.  
\end{lem}
This is not a bijection since some elements of $\HC$ do not encode any paths. 

\begin{proof} Let us take an element $h=(h_u,u\in V)$ in the image, and let us reconstruct its unique preimage $w$ by the map $\Heap$ (in fact, we build $\ra{w}$). The total number of edges in $h$ is the number of steps $m$ of any walk $w$ such that $\Heap(w)=h$. If $m=0$, the statement is clear. Recall that the collection of heaps gives the steps of $\ra{w}$. Assume $m>0$.
 This allows to recover the only possible value for the final point $w_m$ too: this is the unique node $w_m$ having one more outgoing edge than incoming (if such a node does not exist, the path ends and starts at $r$,
  and then $w_m=r$). Now, knowing the last position of the walk $w_{m}$, take $(w_m,u)$ the rightmost (top see \Cref{heaps}) edge in $\Heap_{w_m}$, record the identity of the node $u$, and remove the edge $(w_{m},u)$ from $\Heap_{w_m}$. Now, set $w_{m-1}:=u$. Iterate the construction until all the heaps are empty.\\
  The statement about the weights easily follows.
\end{proof}

\begin{defi}\label{defi:heapc}
  \bir
     \itr A collection of heaps $H\in \HC$ with a balanced passport  i.e.
 \ben\label{eq:dqsuld2}
 \OUTIN\l(H\r)= [ (N_u,u \in V),(N_u,u \in V)],
 \een for some $N_u\geq 0$, is called a heap of cycles. 
 Denote by ${\sf HCycles}$ the set of heaps of cycles and by ${\sf HCycles}({\sf free}\ni v)$ the heaps of cycles where $N_v=0$, meaning that $v$ is not an element of any cycle.
 \itr A  heap of cycles $H\in {\sf HCycles}$ satisfying $N_u\leq 1$ for all $u$, is called trivial (it is a set of non intersecting cycles). We denote by ${\sf HCycles}^{{\sf Trivial}}({\sf free}\ni v)$
 the set of trivial heap of cycles for which $N_v=0$ (that is, non incident to $v$).\eir
\end{defi}~

\begin{rem} It is folklore to build ``classical heap of cycles'' from ``our collection of heaps with balanced passport'' see Viennot \cite{VX}. Classical heaps of cycles are equivalence classes of sequence of cycles $(c_1,\cdots,c_p)$ (were $p\geq 0$),  where the relation is the transitive closure of the following relation: two sequences $(c_1,\cdots,c_p) \sim_R (c'_1,\cdots,c'_{p'})$ are equivalent if they coincide except for some non intersecting and consecutive cycles $c_i$ and $c_{i+1}$ and $(c_i,c_{i+1})=(c'_{i+1},c'_i)$. In words: non intersecting cycles commute, and intersecting cycles do not.
  Following Cartier \& Foata \cite{CF} terminology, sequences of cycles with the partial commutation formula form a partially commutative monoid (for the concatenation operation), which can be geometrically represented using heap of ``real'' cycles (see Viennot \cite{VX} and Krattenthaler \cite{CK}).\par
  Giving these classical heap of cycles one can form the collection of heaps $\Heap(c_1)\oplus\cdots\oplus\Heap(c_p)$. \\
  To define the converse we need to introduce first, the ${\sf PopCycle}$ function. Assume that a collection of heaps $(H_u,u\in V)$ is given. For a vertex $w_0$ we define the procedure ${\sf PopCycle}(w_0)$ as: follow the path starting at $w_0$ and stop it the first time it forms a cycle $c$, that is the path $w_0,w_1,\cdots,w_j$ satisfies $w_i=w_j$ for some $i<j$ and $j$ is the first index with this property; then return $c$ the cycle formed by $(w_i,\cdots,w_j)$ and suppress the cycle from the heaps (that is all the edges $(w_i,w_{i+1}),\cdots,(w_{j-1},w_j)$). 
  
  Now, to define the converse do the following: if all the heaps of the collection are empty, then associate the empty heap of cycles; otherwise, consider $w_0=u$ a vertex with non-empty heap. Starting from $i=1$ repeat the following procedure until the heaps are all empty: define $c_i:= {\sf PopCycle}(w_0)$, increment $i$ and define $w_0=u$, for a $u$ chosen among those for which $H_u\neq \varnothing$ (see \Cref{Cycleheaps} for a representation of this map).
  This procedure is well-defined since passports are balanced, i.e. the number of incoming edges is equal to the number of out-coming edges; and this property is preserved each time ${\sf PopCycle}$ is used. At the end, the sequence $(c_1,\cdots,c_m)$ obtained depends on the order of the nodes $u$ chosen, but it can be proved that for all choices of $u$, the final sequences of cycles are equal as heap of cycles.
  \begin{figure}[h!]
  	\centering
  	\includegraphics[scale=0.35]{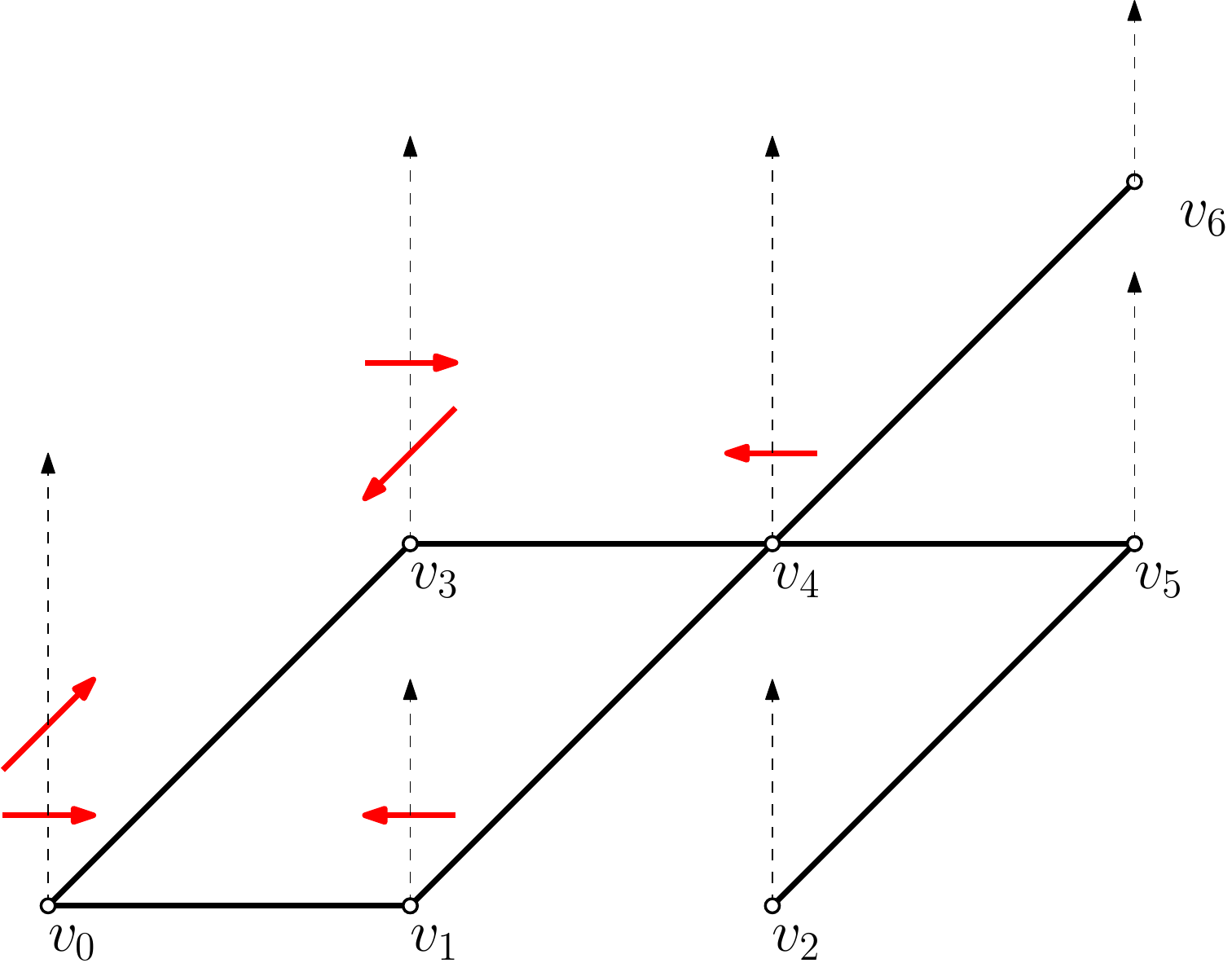} \hspace{5pt} \includegraphics[scale=0.35]{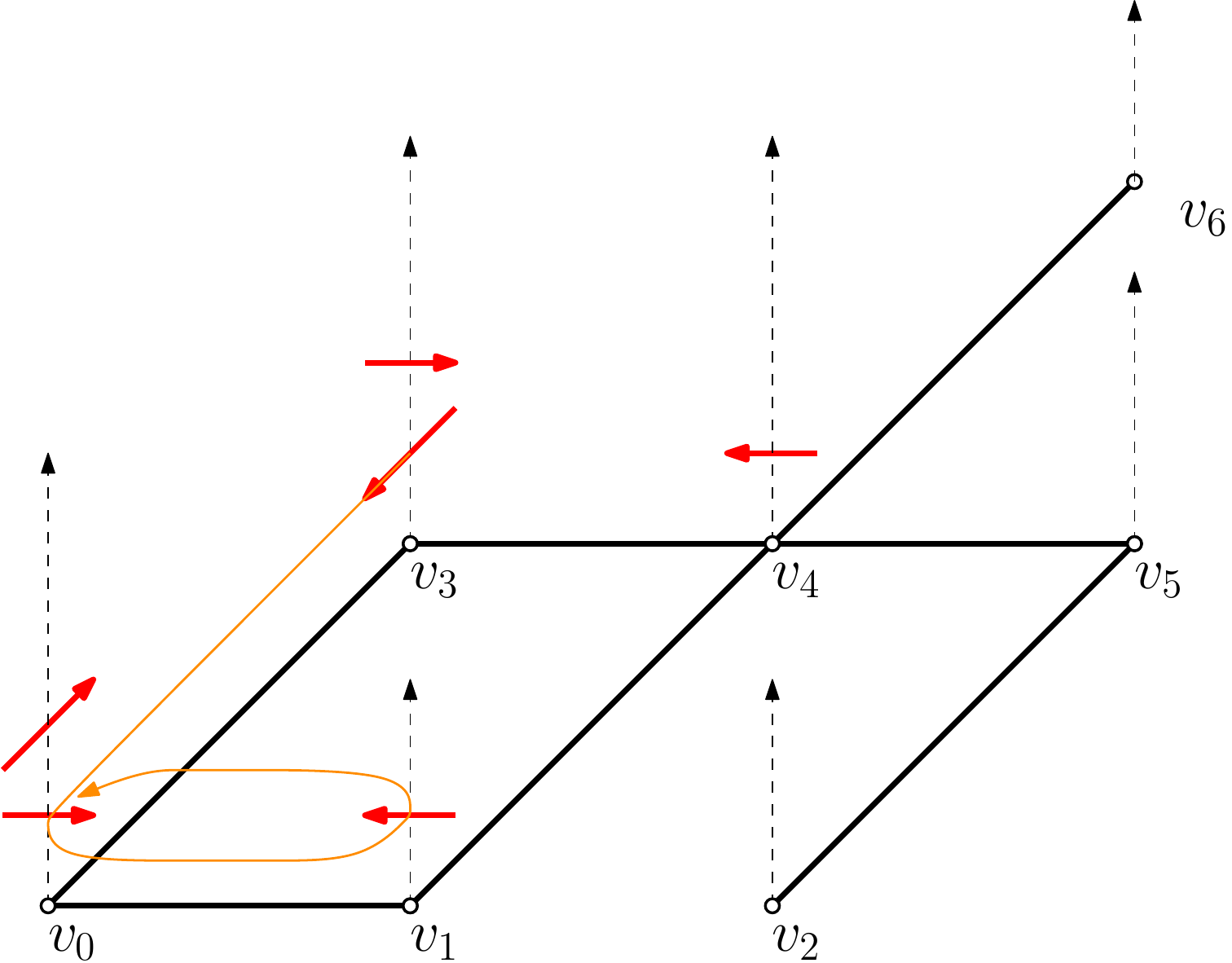} \hspace{5pt} \includegraphics[scale=0.35]{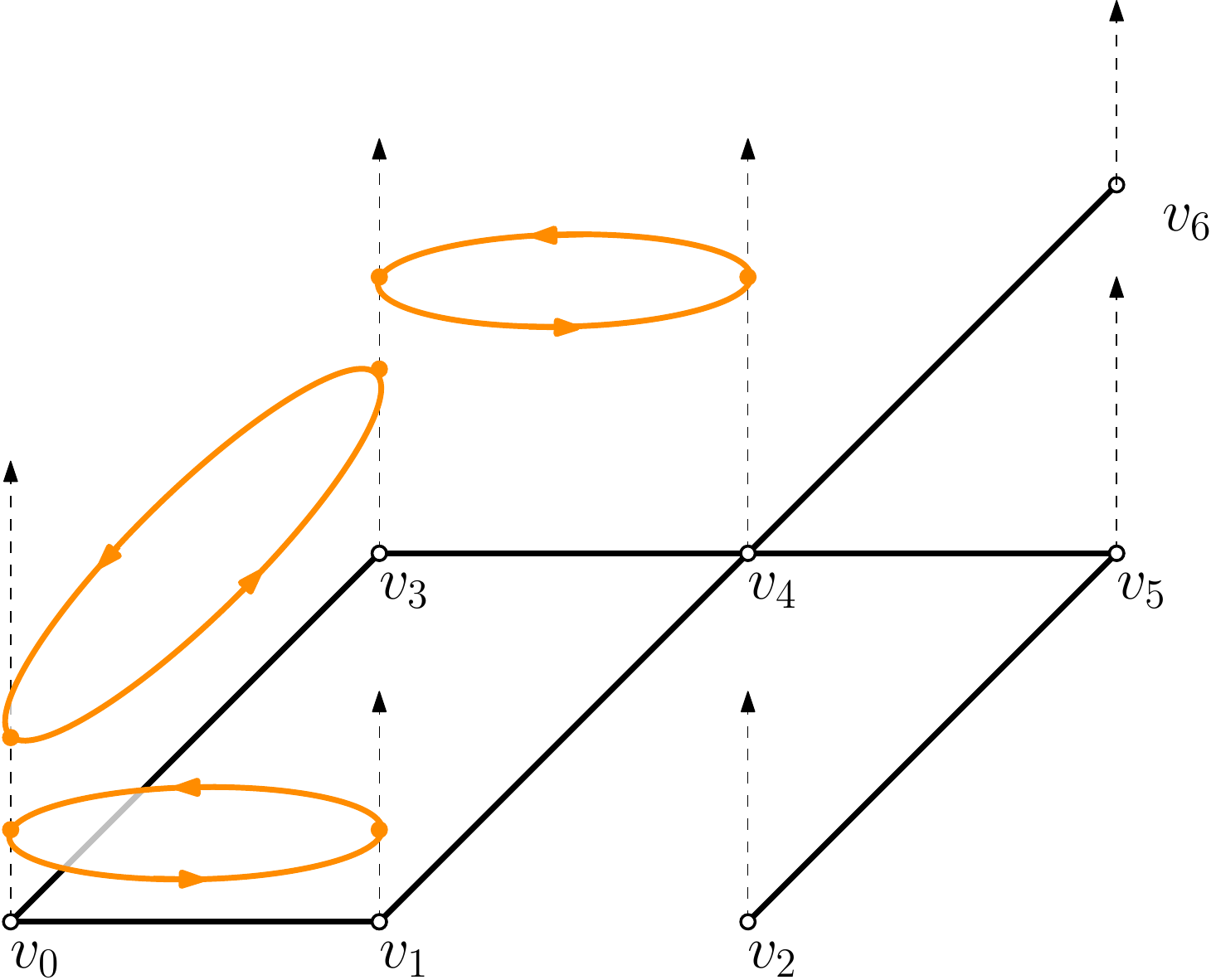}
  	\caption{Left: representation of a heap of cycles as in \Cref{defi:heapc}. Center: algorithm to form the classical heap of cycles started from $v_3$. Right: representation of a classical heap of cycles (there are some cycles on top of the others). Cycles here have length 2, but they may be of any length in general.}
  	\label{Cycleheaps}
  \end{figure}
\end{rem}
For a trivial heap of cycles $H=( H_u,u\in V)$, define
\be
{\sf SignedWeight}(H)= {\sf Weight}(H)(-1)^{|{\sf Cycles}(H)|}.
\ee
We have the following inversion formula
\ben\label{eq:qdq2}
\sum_{H\in {\sf HCycles}({\sf free}\ni f)}{\sf Weight}(H)= \left(\sum_{H\in{\sf HCycles}^{{\sf Trivial}}({\sf free}\ni f)} {\sf SignedWeight}(H)\right)^{-1} \stackrel{\eref{eq:fqs}}{=} \det\l(\Id-M^{(f)}\r)^{-1}.
\een
This formula is famous in combinatorics (see \cite{VX,CK}). A simple proof : we will prove that the product of the sums of the two first members of \eref{eq:qdq2} gives 1. For this associate to each pair $(h,t)$ of heap of cycles and trivial heap of cycles  (an element of ${\sf HCycles}({\sf free}\ni f)\times {\sf HCycles}^{{\sf Trivial}}({\sf free}\ni f)$),  a pair $(h',t')$ of the same type obtained as follows: Take the greatest cycle $c$ in the union of $t$ and of the set of cycles of $h$ having no cycle above them and which moreover do not  intersect the cycles of $t$. Now, to obtain $(h',t')$ from $(h,t)$, remove $c$ from the component which contains it, and move it to the other (if it is move on $h$, it is placed on top of it). This operation is an involution which changes the sign of the weight, so that the total weight of the pairs $(h,t)$ coincides with the weight of the single pair $(\varnothing,\varnothing)$ which is 1.

\subsection{Step 2. Collection of heaps of a covering path corresponding to a given tree}

Consider the subset of $\Paths(r)$, 
\[\Paths(t,r,f):=\{w \textrm{ covering }: \FET(w)=(t,r), \underline{w}=r, \overline{w}=f\}.\]
The passport of any path $w \in  \Paths(t,r,f)$, satisfies,
\ben\label{eq:fejfks}
\OUTIN(\Heap(w))= [ (N_u+1_{u=f},u \in V),(N_u+1_{u=r},u \in V)],
\een
for  $N_{f}=0$, some $N_r\geq 0$, and some $N_u \geq 1$ for all $u\notin\{f,r\}$.
\par
From the $\FET$ construction, the last visited vertex is a leaf:
\begin{lem}
  If $f$ is an internal node of $t$, then $\Paths(t,r,f)=\varnothing$.
\end{lem}
All paths $w\in \Paths(t,r,f)$ have a step from $p(u)$ to $u$, for each node $u\neq r$. The oriented edges of $(t,r)$ can be seen as a collection of heaps $H^t=(H_u^t, u \in V)$ defined by:
\be
\left\{\begin{array}{ccl}H_u^t &=& [(u,p(u)], \textrm{ for }u \neq r,\\
H_{r}^{t}&=& \varnothing, \end{array}\right.
\ee 
whose passport is 
 \ben\label{eq:dqsdfqs}
 \OUTIN(H^t)=\l[ \l(1-1_{u=r}, u \in V\r), \l(\deg{(t,r)}{u}, u \in V\r)\r].
 \een 
By construction, the edges of $H^t$ are the first edges of the collection $\Heap(w)$ (we put them according to their order in $w$), and then can be removed using the remove prefix operation; it produces the collection of ``truncated heaps'' defined by removing the edges of $t$ from $\Heap(w)$ 
\ben\label{eq:dec}
\Heap^{\sf Trunc}(w):= \Heap(w)\ominus_{PR}H^{t}(w).
\een
If \eref{eq:fejfks} holds, in view of \eref{eq:dqsdfqs}, the passport of $\Heap^{\sf Trunc}(w)$ is
 \ben\label{eq:dqsuld}
 \OUTIN\l(\Heap^{\sf Trunc}(w)\r)= [ (N_u+1_{u=f}-(1-1_{u =r}),u \in V),(N_u+1_{u=r}-\deg{(t,r)}{u},u \in V)].
 \een
 One of the main point of the proof is the following Lemma, which says that we completely know the set of collections of {\bf truncated} heaps  $\Heap^{{\sf Trunc}}\l[  \Paths(t,r,f)\r]$.
 In words, they correspond, prior to truncation, to all collections $\Heap(w)$ satisfying \eref{eq:fejfks} and large enough so that $H^t$ can be removed.
Set
\be
\HC[ (o_u,u\in V),(i_u,u\in V)]:=\{h \in \HC: \OUTIN(h)=[(o_u,u\in V),(i_u,u\in V)]\}
\ee
\begin{lem}\label{lem:key}
  If $f\in \partial t$, the set $\Heap^{{\sf Trunc}}\l[  \Paths(t,r,f)\r]$  is equals to 
  \ben\label{eq:qdsdhtegzrf}
  \Xi(t,r,f):=\bigcup \HC\l[(n_u + (\deg{(t,r)}{u}-1)\1_{u\in t^o} , u \in V), (n_u+ \1_{u \in \partial t}(1-1_{u=f}), u \in V)\r]\een 
  where the union is taken on all $n_u\geq 0, \forall u\neq f$, and $n_{f}=0$ (no incoming or outgoing edge at $f$, while $n_u+\1_{u \in \partial t}(1-1_{u=f})>0$ for the other leaves).
\end{lem}
Lemma \ref{lem:key} describes the set of heaps that are not circled in \Cref{heaps} (in all generality).
\begin{proof} \bls Proof of $\Heap^{{\sf Trunc}}\l[  \Paths(t,r,f)\r]\subset \Xi(t,r,f)$. This inclusion can be seen thanks to a  change of variables from \eref{eq:dqsuld}. Indeed, from
	   \eref{eq:dqsuld} we must have $N_u-\deg{(t,r)}{u}+1_{u=r}\geq 0$ and $N_u+1_{u=f}-1+1_{u =r}\geq 0$, which implies
            \[N_u \geq \l(\deg{(t,r)}{u}-1_{u=r}\r)  \1_{u\in t^o}+\1_{u\in \partial t} (1-1_{u=f}) \]
           since  $\deg{(t,r)}{u}\geq 1$ when $u\in t^o$ and $\deg{(t,r)}{u}=0$ when $u\in\partial t$.
    We then set
    \[n_u:= N_u-\l[ \l(\deg{(t,r)}{u}-1_{u=r}\r)  \1_{u\in t^o}+\1_{u\in \partial t} (1-1_{u=f})\r]\]
      and compute the passport \eref{eq:dqsuld} after this change of variable to conclude this inclusion.\\
\bls Now, take $h=(h_u, u \in V)\in \Xi(t,r,f)$ and consider
 $H:=H^t\oplus h$ the potential unique collection of heaps such that  $\Heap^{{\sf Trunc}}\l[H\r]=h$.
 Let us prove that there exists a path $w\in \Paths(t,r,f)$ such that $\Heap(w)=H$.
 As explained in the proof  of Lemma \ref{lem:ujkgkuj}, the collection of heaps $H$ can be used to construct a path: denote by $p$ the current point and $\ra{w^\star}$ the current path which  are initialized as $p=f$ and $\ra{w^\star}=(f)$.\\
 $(a)$ If $H_p$ is empty then stop the construction and ``return $\ra{w^{\star}}$''. If $H_p$ is not empty, consider the last edge $(p,u)$ of the heap $H_p$.\\
 $(b)$ Remove $(p,u)$ from $H_p$, add $(p,u)$ as last step in $\ra{w^\star}$. \\
 $(c)$ Set $p=u$ and go to $(a)$.

 This construction stops when a vertex $p$ with an empty heap $H_p$ is reached. To conclude the proof, taking into account Lemma \ref{lem:ujkgkuj}, it suffices to prove that all the heaps are empty when this event occurs (since it is necessary and sufficient for $H$ to encode a path). \par
  We claim that indeed, all the heaps are empty and moreover $p=r$.
 Observe first that the global passport of $H$ is obtained by doing $\OUTIN(H^t)+\OUTIN(h)$,
 which gives for some $(n_u, u\in V)$, 
 \be
 &&\left[(n_u + (\deg{(t,r)}{u}-1)\1_{u\in t^o}+1-1_{u=r} , u \in V),(n_u+ \deg{(t,r)}{u}+\1_{u \in \partial t}(1-1_{u=f}), u \in V)\right] \\
 &=&\left[(n_u +\deg{(t,r)}{u}+ 1_{u\in\partial t}-1_{u=r} , u \in V),(n_u+ \deg{(t,r)}{u}+\1_{u \in \partial t}-1_{u=f}, u \in V)\right].\ee 
Up to a change of variables, this matches the passport of a path (see \eref{eq:fejfks}); however  \eref{eq:fejfks} does not characterize collection of heaps corresponding to paths, but collection of heaps of paths concatenated with heap of cycles (see Defi. \ref{defi:heapc} and \eref{eq:dqsuld2}). We then need to discard the possibility of cycles.
Notice that each vertex different from $r$ has as many outgoing as incoming edges, so that necessarily $p=r$.
We will prove that if the heap $H_v$ of an internal node $v$ is empty, then it is also the case of the heaps of its children. Indeed, if $H_{v}$ is empty, then, all the outgoing edges out of $v$ have been used, and therefore all the incoming edges to $v$ must have also been used too! (one always needs to enter before exiting). Therefore when $H_v$ is empty all the incoming edges arriving at $v$ have been used in their respective heaps. As a matter of fact, the leftmost edge of the heap $H_u$ for each child $u$ of $v$, is the edge $(u,v)$, and this edge is incoming at $v$: as a conclusion when $H_v$ is empty, $H_u$ is empty too.
\end{proof} 
By Lemma \ref{lem:key},
\begin{cor}\label{eq:cor9} For any $(t,r)\in \SP^{\bullet}(G)$,
	\be
	\P(\FET(w_0,\cdots,w_{\tau_{|V|}(w)}) =(t,r)) &=& \W(H^t) \sum_{f \in \partial t} \rho_f \W(\Xi(t,r,f))\\
	&=& \Bigl(\prod_{e\in E(t,r)} \ra{M_e}\Bigr) \sum_{f \in \partial t} \rho_f\W(\Xi(t,r,f)).
	\ee
\end{cor}

\subsection{The combinatorial trick}
\label{sec:TCT}
In the preceding sections we use collections of heaps to present some operations that could have been presented without them, since withdrawing the steps of the first entrance trees could have been done directly on $w$. From here, collections of heaps becomes the central objects: just consider a moment the collections of heaps $\Heap^{\sf Trunc}(w)$ which comes from $w$, and have been defined using $t$. From now on, the economic way to think about them is to forget where they come from ! They are just elements $(H_u,u\in V)$ of some sets of collections of heaps whose passports satisfy some constraints (see Proposition \ref{pro:heapinj}). When seen as coming from $w$ and $\ra{w}$, it is tempting to equip the edges of $(H_u,u\in V)$ with a global order inherited from $w$, but this way of thinking drove us in some dead ends. \par
Collections of heaps are simple object, devoid of global ordering, even if each heap comes with a fixed order. In the sequel, in order to make clear that we don't try to follow in any way the order induced by $w$, we still call ``first edge'' of $H_u$ the leftmost one (the bottom on the figures in the representation of heaps, for example, \Cref{heaps}). All our algorithms will use ``first edges first'': hence since the edges of $t$ have been removed, following paths constructed using the first edges of the heaps produce paths that do not correspond to portions of $w$ or of $\ra{w}$ in general.\\
The general principle is that one can do many things with collections of heaps. Including playing golf?  Yes.

\subsection{Combinatorial golf sequences decomposition}

In the sport called \textit{golf}, players hit balls into a series of holes. 
We introduce {\bf the combinatorial golf sequence}, a new combinatorial object: it consists in a sequence of paths on $G$, modeling the moves of some balls (placed on some vertices) on a golf court (the graph $G$), until each of them finds a different \textit{empty hole} (some special vertices). It is parameterized by the pair $\Big[\HoleSet, S\Big]$ as follows: \\
\tb~ $S=(S_1,\cdots,S_{\Nbb})$  is the {\it sequence of starting vertices} from which balls are to be played,
with $S_j$, being the starting position of  the  $j$-th. 
We denote by  $\Nb_u:=|\{j: S_j=u\}|$ the number of balls starting at $u$, $\Nbb=\sum_u \Nb_u$, the total number of balls;  $(\Nb_u,u\in V)$ is called the  {\it counting sequence}  of $S$.\\
\tb~ $\HoleSet$ is a subset of $V$, the set of empty holes, and we require $u \in \HoleSet \imp \Nb_u=0$ so that a node which is the ending point of a path, is not the starting of any paths. 
\begin{defi}\label{def:fdqs}
A sequence of paths $(w^{(1)},\cdots,w^{(|\Nb|)})$ is an element of ${\sf GolfSequence}\Big[\HoleSet, S\Big]$ iff:\\
-- for any $j$, $w^{(j)}=\l(w^{(j)}_i, 0\leq i \leq \l|w^{(j)}\r|\r)$ is a finite path starting at $S_j$, i.e. $w^{(j)}\in \Paths(S_j)$, \\
-- the final positions of each path is a hole, meaning that
\[{\sf Final}(i):=w^{(i)}(|w^{(i)}|)\in \HoleSet, \textrm{ for all } 1\leq i \leq \Nbb.\]
-- each hole captures the first ball visiting it, and looses after that the property of being a hole. It becomes a standard vertex for subsequent paths. This means that the vertices ${\sf Final}(i)$ are distinct, and:\\
  \tb~ the time $|w^{(1)}|$ is the hitting time of $\HoleSet$ by $w^{(1)}$,\\
  \tb~ ... the time $|w^{(k)}|$ is the hitting time of $\HoleSet\setminus \{{\sf Final}(1),\cdots,{\sf Final}(k-1)\}$ by $w^{(k)}$, for each $k$.
\end{defi}
\begin{rem}
  ${\sf GolfSequence}\Big[\HoleSet, S\Big]$ is empty if $\Nbb> |\HoleSet|$, and non-empty if $\Nbb\leq |\HoleSet|$.
  \end{rem}

\centerline{------------------------------------}
\paragraph{The stochastic golf sequence.} We call ${\sf StochasticGolfSequence}\Big[\HoleSet, S\Big]$, a random variable $X:=(X^{(1)},\cdots,X^{(\Nbb)})$ taking its values in ${\sf GolfSequence}\Big[\HoleSet, S\Big]$ distributed as follows. For any $j\in \{1,\cdots,\Nbb\}$, conditionally on $\{X^{(1)},\cdots,X^{(j-1)}\}$, the trajectory of the $j$th ball $X^{(j)}:=(X_t^{(j)},t\geq 0)$ is a Markov chain with kernel $\ra{M}$ started at $S_{j}$ and killed at the first element of $\HoleSet$ not present in the already killed trajectories $(X^{(k)},k\leq j-1)$.\\
\centerline{------------------------------------}\\
The weight of a combinatorial golf sequence $(w^{(1)},\cdots,w^{(\Nbb)})$ is set as
\be\W\l[w^{(1)},\cdots,w^{(\Nbb)}\r]:= \prod_{j=1}^{\Nbb} \W\l(w^{(j)}\r)=\prod_{j=1}^{\Nbb} \prod_{i=0}^{\l|w^{(j)}\r|-1} \ra{M}_{w_{i}^{(j)},w_{i+1}^{(j)}}\ee
which then corresponds to the probability of independent trajectories in the stochastic golf sequence using the kernel $\ra{M}$, killed at the first time it visits a still available hole.

Denote by ${\sf GolfSequence}\Big[ \HoleSet, S, {\sf free}\ni f\Big]$ the subset of ${\sf GolfSequence}\Big[ \HoleSet,S\Big]$ leaving the hole $f$ empty, meaning that $f\not\in\{{\sf Final}(i),1\leq i \leq \Nbb\}$).
\begin{pro}\label{theo:ts}[Trivial stochastic golf sequence property]~\\
  Let $X=(X^{(1)},\cdots,X^{(\Nbb)})$ be a ${\sf StochasticGolfSequence}\Big[\HoleSet, S\Big]$ with $\Nbb\leq |\HoleSet|$. 
  \bir\compact
  \itr We have  $\P\l(X \in {\sf GolfSequence}\Big[\HoleSet, S\Big]\r)=1$.
  \itr If $\Nbb=|\HoleSet|-1$ (eventually a single hole will remains free with probability 1) then
  \[\sum_{f \in {\HoleSet}} \P\Big(X \in {\sf GolfSequence}\Big[ \HoleSet, S, {\sf free}\in f\Big]\Big)=1.\]
  \eir
  \end{pro}
  \begin{proof} 
  	When $\Nbb\leq |\HoleSet|$, the stochastic golf sequence is well defined and takes its value in  ${\sf GolfSequence}\Big[\HoleSet, S\Big]$: the paths are a.s. finite by irreducibility of the Markov chains. The second statement follows since a single hole must remain free at the end.
    \end{proof}

\begin{pro} \label{pro:heapinj}For any parameters $(\HoleSet,S)$ such that $u\in \HoleSet\imp \Nb_u=0$,
  the map
  \ben\label{eq:injse}\app{\Heap}{{\sf GolfSequence}\Big[  \HoleSet, S\Big]}{\HC}{\l(w^{(1)},\cdots,w^{(\Nbb)}\r)}{\Heap\l(w^{(1)} \r)\oplus \cdots \oplus \Heap\l(w^{(\Nbb)}\r)}\een
  is a weight preserving injection, and the image is included in
  \ben\label{eq:dqhet}\bigcup_{A:|A|=\Nb\atop{ A\subset \HoleSet}} \bigcup_{n_u\geq 0} \HC\l[(n_u + \Nb_u , u \in V),(n_u+ \1_{u \in A}), u \in V)\r]\een
\end{pro}
\begin{rem} This Proposition is also a key point of our approach. The map
  $\l(w^{(1)},\cdots,w^{(N)}\r)\mapsto\Heap\l(w^{(1)}\r)\oplus \cdots \oplus \Heap\l(w^{(N)}\r)$ is not an injection when defined on a more general set of $\Nbb$-tuple of paths; but for golf sequences, it is.
\end{rem}  
\begin{proof} 
  The fact that the image of any golf sequence is included in \eref{eq:dqhet} comes from the fact that for each $u$, $\Nb_u$ paths start at $u$. The result is that, eventually, a subset $A$ of the $\HoleSet$ with cardinality $\Nbb$ is eventually occupied. This gives a collection of heaps in the set \eref{eq:dqhet}. \\
		Take $H:=\Heap\l(w^{(1)}\r)\oplus \cdots \oplus \Heap\l(w^{(\Nbb)}\r)$ an image under this map. It suffices to explain how to recover $w^{(1)}$ from $H$ (and iterate). Recall that $S_1$ and $ \HoleSet$ are known data. To recover $w^{(1)}$ follow steps $(a)$, $(b)$ and $(c)$ in \Cref{lem:key} started at $p=S_1$ with the difference that we stop and return the path the first time $p\in \HoleSet$ and then we define ${\sf Final}(1):=p$. Do the same for $w^{(k)}$, except that the stopping time is now the first time the constructed path hits $\HoleSet \setminus\{ {\sf Final}(1),\cdots,{\sf Final}(k-1)\}$.  At the end, the set of final points $A$ is the set of final points of the golf paths.
  \end{proof}

\subsection{Decomposition of Truncated heaps}

\begin{pro}\label{pro:qklsd} Consider $(\Nb_u,u \in V)$ some non negative integers, $\HoleSet$ a subset of $V$, such that
\be
 |\HoleSet|=\Nbb+1,\textrm{ and such that } u \in \HoleSet \imp \Nb_u=0,
 \ee
 i.e. after the golf sequence a single vertex remains free.
 Finally let $S=(S_1,\cdots,S_{\Nbb})$ be a sequence of starting points with counting vector $(\Nb_u,u\in V)$. Assume that $f \in \HoleSet$.
   There is a weight preserving bijection between $\mathcal{N}:={\Heap}{\sf GolfSequence}\Big[  \HoleSet, S, {\sf free}\ni f\Big]\times{\sf HCycles}({\sf free}\ni f)$ and
 \be   &&\mathcal{M}:=\bigcup_{n_u\geq 0\atop{n_f=0}} \HC\l[(n_u + \Nb_u , u \in V),(n_u+ \1_{u \in \HoleSet\setminus\{f\}}), u \in V)\r].
   \ee
\end{pro}
See \Cref{heaps2} for an application of this proposition.
\begin{proof}
The map
  $\phi:(H_1,H_2)\mapsto H_1\oplus H_2$ is an injection from $\mathcal{N}$ to $\mathcal{M}$ since passports are compatible, and as following the idea in the proof of Proposition \ref{pro:heapinj}  (in the case where $|A|=\Nb-1$), it is possible to recover the golf sequence from the concatenation.\\
  Now, $\phi$ is also a  surjection: take an element $H\in \mathcal{M}$ and let us build a pair $(h,h')\in\mathcal{N}$ such that $h\oplus h'=H$.
  For that, we need to show that we can extract from $H$ all the golf trajectories. 
  Consider $S_1$, and follow the path coming out of $S_1$ using the edges given by the collection $H$, as usual. Eventually, this path will reach an element $z$ of $\HoleSet$: indeed, for all nodes which are not in $\HoleSet$, the number of outgoing edges is $\geq$ the number of incoming edges. It is then not possible that the ball arriving at an internal node $u\in t^o$, find an empty heap $H_u$ and is trapped there. \par
  This gives the first golf sequence, and the removal of all the edges used to define it provides an element of $ \bigcup_{n_u'\geq 0\atop{n_f'=0}} \HC\l[(n_u' + \Nb_u -1_{u=S_1}, u \in V),(n_u'+ \1_{u \in \HoleSet\setminus\{f,z\}}), u \in V)\r]$ since the balance at each vertex of this first path is 0, except at the starting and ending position. We are in the same situation as before with one less path in the golf sequence. When all the paths of the golf sequence have been extracted, the possibly remaining heaps are balanced, and still $n_f=0$, so that the remaining steps form a heap of piece of ${\sf HCycles}({\sf free}\ni f)$.
  \end{proof}
\begin{figure}
	\centering
	\includegraphics[scale=0.5]{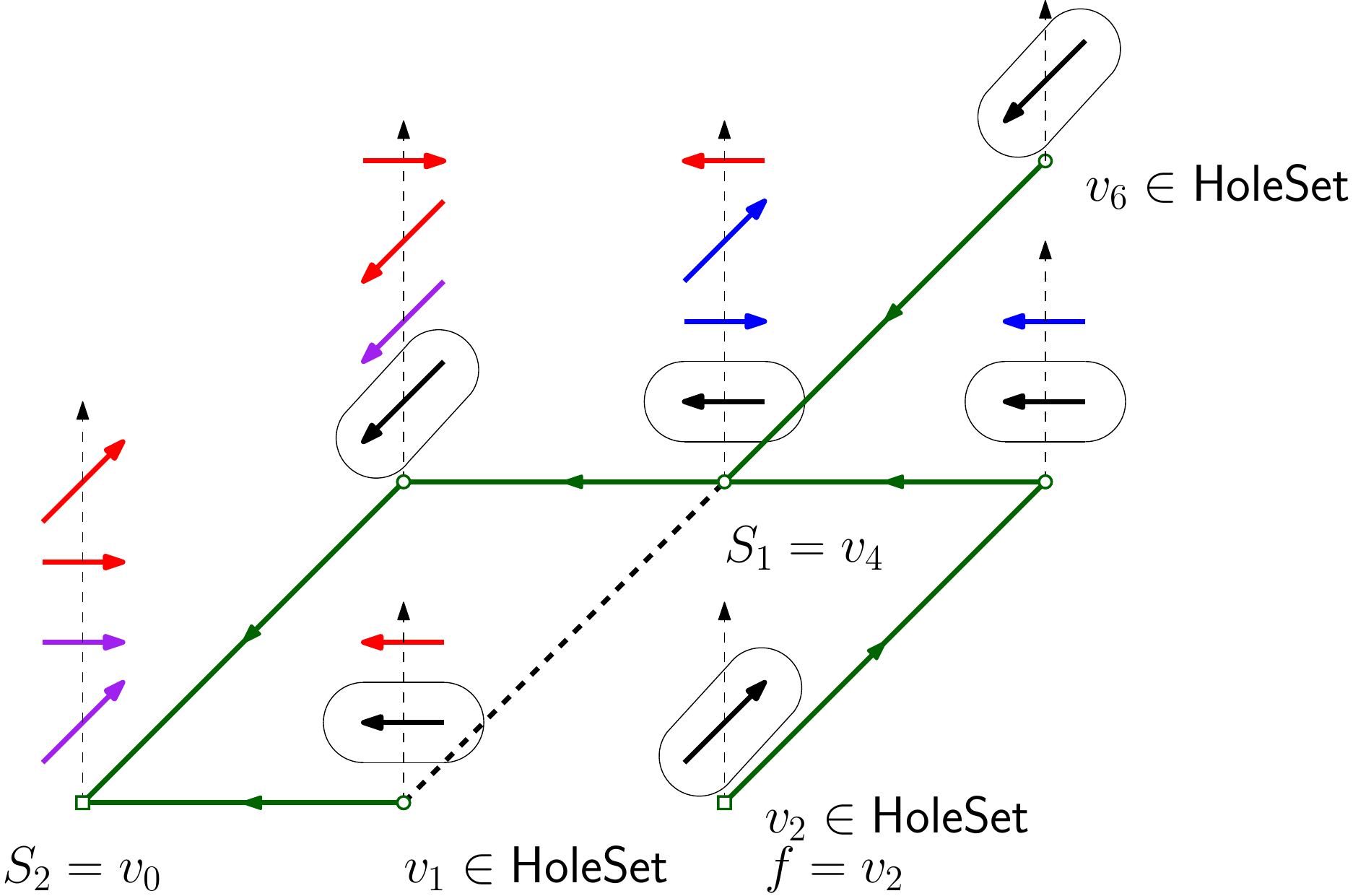}
	\caption{An application of \Cref{pro:qklsd} to our example in \Cref{heaps}. The bijection takes as input a pair of heaps: the heap of cycles is in red, and the golf sequence heap is colored  blue and purple (the blue steps shows the path starting at $S_1= v_4$ and the purple steps the path starting at $S_2=v_0$; in this case $S=(S_1,S_2)=(v_4,v_0)$). The bijection outputs the heap obtained by putting the heap of cycles above the golf sequence heap. The corresponding $\HoleSet$ is $\{v_1,v_2,v_6\}$, and $f=v_2$ (it is in {\sf free}, because it belongs to $\HoleSet$, and there are not incoming edges at $v_2$).
     %lacing about 
      %We present the resulting collection of heaps under the weight preserving bijection, decomposed by colors in order to identify \LF{the Cartesian product originating it:}the heap of cycles in red and a golf sequence heap in blue and purple where $\HoleSet = \{v_1,v_2,v_6\}$, golf sequence $S=(S_1,S_2)=(v_4,v_0)$ and $f=v_2$ (it belongs to ${\sf free}$ since there isn't any incoming nor outgoing blue or purple steps from/to $v_2$). The blue steps shows the path starting at $S_1= v_4$ and the purple steps the path starting at $S_2=v_0$.
      We kept the steps used to construct the tree, which are circled, in order to show the full decomposition of  \Cref{heaps}.}
	\label{heaps2}
\end{figure}

\subsection{Conclusion: proof of Theorem \ref{theo:kgyqsd}}

From \Cref{pro:qklsd} and Lemma \ref{lem:key}, one identifies that
\[\W(\Xi(t,r,f))= \W\l({\Heap}{\sf GolfSequence}\Big[  \HoleSet, S, {\sf free}\ni f\Big]\r)\W({\sf HCycles}({\sf free}\ni f))\]
where $\HoleSet=\partial t$ and $S$ is  any starting sequence with counting sequence $(\Nb_u,u\in V)=((\deg{(t,r)}{u}-1)\1_{u\in t^o},u \in V)$: in words $\deg{(t,r)}{u}-1$ balls start at the internal node $u$.
Such a starting sequence ($S^{i},1\leq i \leq \Nbb)$ is easy to build: take any order on $t^o$, and sort the nodes $u_1,\cdots,u_{|t^o|}$, and take as sequence $S$: the node $u_1$ repeated $\deg{(t,r)}{u_1}-1$ times, followed by $u_2$ repeated $\deg{(t,r)}{u_2}-1$ times, etc. \\
Recall Corollary \eref{eq:cor9}.
We want to prove that
\[\sum_{f \in \partial t} \rho_f\W(\Xi(t,r,f))=\frac{1}{\sum_{x}\det(\Id-M^{(x)})}.\]
Taking into account that $\rho_f=\frac{\det(\Id-M^{(f)})}{\sum_{x}\det(\Id-M^{(x)})}$ and \eref{eq:qdq2}
we just need to prove that
\[\sum_{f\in \partial t} \W({\Heap}{\sf GolfSequence}\Big[  \HoleSet, S, {\sf free}\ni f])=1\]
but this is precisely what says the  trivial stochastic golf sequence property (\Cref{theo:ts}$(ii)$).
\begin{rem} In Section \ref{sec:TCT} we said that we would use the collection of heaps $\Heap^{\sf Trunc}(w)$ while forgetting the initial order of edges coming from $w$ or $\ra{w}$: we used these collections of heap to define the golf sequences, and the additional heap of cycles, and none of them have been defined using the order of the original steps of $w$ or $\ra{w}$. \par
    However, using portions of $w$ (or $\ra{w}$) to define golf trajectories seem possible but brings many complications:  because somehow, ``the next ball'' to be played depends on the arrival position of the previous. In terms of heaps of cycles everything become messy; for example the global heap of cycles produced  finally in our construction, needs to be partially piled after each golf path, if one wants to use the chunks of consecutive steps of $\ra{w}$ in their initial order. These considerations make this ensemble quite difficult to follow and would produce a much longer and involved argument.
\end{rem}

\begin{rem} In the proof, we have rearranged the steps of the random walk with kernel $\ra{M}$. What we did -- notably around the notion of golf sequences -- is reminiscent of Diaconis \& Fulton paper \cite{DF} and the commutative property they proved (this property being also at the core of the study of internal DLA). When one studies a sequence of stopped random paths, where the stopping time of each path depends on the ending points  of the preceding paths (the starting points being fixed, and possibly different), under some various conditions, the distribution of the set of final points of the trajectories is independent from the order of the sequence of paths and the distribution of the vector giving the total number of traversals of each edge by the sequence of paths, is also independent from the order of sequence of path. 
  \end{rem}

\section{Proof of Theorem \ref{theo:kgyqsd} by adaptation of the classical argument}
\label{sec:qdqsd}
The proof follows mainly, Aldous--Broder's ideas, except that $M=\ra{M}$ is not an hypothesis; we proceed here and there to some small changes with respect to classical arguments.

\paragraph{The $\LET$.}
Any path $(z_0,\cdots,z_n)$ on $G$ can be used to define a rooted tree
\[(\Gamma_n,z_n):=\LET(z_0,\cdots,z_n),\]
rooted at $z_n$, as follows: set $(\Gamma_0,z_0)$ as the tree reduced to its root $z_0$; from $k=0$ to ${n-1}$, construct $(\Gamma_{k+1},z_{k+1})$ from $(\Gamma_{k},z_k)$ by the addition of the edge $(z_k,z_{k+1})$ and the suppression of the outgoing edge from $z_{k+1}$ (if any). The set of nodes of $(\Gamma_n,z_n)$ is $\{z_0,\cdots,z_n\}$. For any $k\in \{1,\cdots,|\{z_0,\cdots,z_n\}|\}$, denote by
	\[\nu_k=\max\{j : |\{z_j,\cdots,z_n\}|=k\},\]
the last time $k$ nodes are to be visited ``in the future'', which is also a last visit time for a node. It follows that the tree $\LET(z_0,\cdots,z_n)$ has root $z_n$ and has for edges 
\be
(z_{\nu_k}, z_{\nu_k+1}), \textrm{ for } k=|\{z_0,\cdots,z_n\}| \textrm{ to }2.
\ee

In  Definition \ref{defi:dqsd}, we defined the $\FET$ associated with a covering path; this definition can be extended to any path, covering or not, killed at the covering time or not.
The following lemma is proved by a straightforward checking:
\begin{lem}\label{lem:qsdqs} For any path $(w_0,\dots, w_n)$ on $G$: 
  \ben\label{eq:qsdqs}
  \FET(w_0,\cdots,w_n)=\LET(w_n,\cdots,w_0).
  \een
\end{lem}
Several remarks follow:\\
$\bullet$ The sequence of trees $\l( \FET(w_0,\cdots,w_k) ,0 \leq k \leq n\r)$ is non decreasing, and the jumps of this sequence correspond to the $\tau_m=\inf\{j~: |\{w_0,\cdots,w_j\}|=m\}$ where a new node is attached to the tree. Moreover, for a Markov chain $(W_i,i\geq 0)$,
\[ \FET(W_i, i\geq 0) = \FET(W_i, 0\leq i \leq \tau_{|V|}).\]
$\bullet$ The right hand side in \eref{eq:qsdqs} is somehow not natural since we reverse the time (time is decreasing). 
The following Lemma is a direct consequence of the previous one
\begin{lem}\label{lem:qdsdq}  Assume that $(X_k, k\geq 0)$ and $(Y_k, k \leq 0)$ are two processes such that,
  \ben\label{eq:inversion}
  (X_0,X_1,\cdots)\eqd (Y_{0},Y_{-1},\cdots).
  \een
  For
  \be
  \tau_m&=&\inf\{j: |\{X_0,X_1,\cdots,X_j\}|=m\},\\
  L_m &=& \inf\{j:|\{Y_0,Y_{-1},\cdots,Y_{-j}\}|=m\},
  \ee
  we have
  \ben\label{eq:retournement}
(X_0,\cdots,X_{\tau_m})\eqd (Y_{0},\cdots,Y_{-L_m})
\een
and then
\ben\label{eq:qdsa1}
\FET(X_0,\cdots,X_{\tau_m})= \LET(X_{\tau_m},\cdots,X_0)\eqd \LET(Y_{-L_m},\cdots,Y_0).\een
\end{lem}The equality in law \eref{eq:inversion} is crucial: assume that $(X_k, k \geq 0)$ is a Markov chain with kernel $M$; in this case the initial value $X_0$, play a role for $(X_k, k\geq 0)$, which is really different from the final value $Y_0$ in $(Y_{-k}, Y_{-k+1},Y_{-k+2},\cdots, Y_0)$. Requiring \eref{eq:inversion} for Markov chains can be done in two natural ways:
\begin{lem}\label{lem:dqsdd}
  $(a)$ If $(X_k,k\in \Z)$ and $(Y_k,k\in \Z)$ are two Markov chains with respective positive Markov kernels $M$ and $\ra{M}$ on $G$, with invariant distribution $\rho$. Assume that both $(X_k,k\in \Z)$ and $(Y_k,k\in \Z)$ are taken under their invariant distributions, then
  \be
  (X_0,X_1,X_2,\cdots) \eqd (Y_0,Y_{-1},Y_{-2},\cdots )
  \ee
  $(b)$ If $M$ is reversible (that is $M=\ra{M}$) then, if  $(X_k,k\in \Z)$ and $(Y_k,k\in \Z)$ are two Markov chains with respective Markov kernel $M$, then for any $x_0$,
  \be
  {\cal L}\l[(X_0,X_1,X_2,\cdots)~|~X_0=x_0\r] = {\cal L}\l[ (Y_0,Y_{-1},Y_{-2},\cdots )~|~Y_0=x_0\r]
  \ee
(so that, somehow, the situation of $(a)$ occurs, conditionally, ``outside the stationary regime'').
   \end{lem}
  Aldous--Broder argument relies on $(b)$ while $(a)$ seems to us more natural and leads to Theorem \ref{theo:kgyqsd}. The following Lemma is not stated by Aldous \cite{Al90} or Broder \cite{Bro89}, but it is a consequence of their proofs.
 \begin{lem}\label{lem:shds} Assume $(Y_{k, k\in \Z})$ is a Markov chain with kernel $\ra{M}$ under its stationary regime, then for every $(t,r)\in \SP^{\bullet}(G)$ one has
 \begin{align}
 \P({\LET}(Y_k,k\leq 0)=(t,r))= {\sf Const}.\prod_{e\in E(t,r)} \ra{M}_e.\label{eq:belle}
 \end{align}
\end{lem}
Before proving the Lemma, let us discuss a bit the consequences of the two last lemmas~:\\
\bls As discussed in Aldous \cite{Al90} and Broder \cite{Bro89}, when $\ra{M}=M$ (reversible case) and $r$ fixed, 
\[{\cal L}((X_0,\cdots,X_{\tau_m})~|~X_0=r)= {\cal L}((Y_{0},\cdots,Y_{-L_m})~|~Y_0=r)\]
so that by Lemma \ref{lem:qdsdq}, 
\be\label{eq:rgsf}
\P(\FET(X_i,0\leq i \leq \tau_{|V|})=(t,r)~|~X_0=r) &=&\P({\LET}(Y_{\leq 0})=(t,r)|Y_0=r)\\ &=& {\sf Const}.1_{x_0=r}\Bigl(\prod_{e\in E(t,r)} M_e\Bigl)/\rho(r).
\ee
\bls If we do not assume $\ra{M}=M$, then for $X$ a Markov chain with kernel $M$, starting under its invariant distribution, we have
\be\label{eq:rfze}
\P(\FET(X_i,0\leq i \leq \tau_{|V|})=(t,r)) &=&\P({\LET}(Y_{\leq 0})=(t,r))\\
&=& {\sf Const}.\prod_{e\in E(t,r)} \ra{M}_e,
\ee
and conditional on $X_0=r$ in the left hand side and $Y_0=r$ in the right hand side, it gives Theorem \ref{theo:kgyqsd}.
\begin{proof}
	{\Cref{eq:belle} comes from the fact that one can view the sequence $({\LET}(Y_k,k\leq m), m \geq 0)$ as a Markov chain with state space $\SP^{\bullet}(G)$ running from $-\infty$\footnote{When the chain is not aperiodic, it has to be seen as the limit when $m\to +\infty$ of a Markov chain starting at $-m$ under its invariant distribution} and therefore under its stationary regime.
	The kernel $Q((t,r),(t',r'))$ of this chain provides the probability to pass from $(t,r)$ to $(t',r')$; this weight is $\ra{M}_{r,r'}$ or 0 depending on whether the root displacement is enough to transform $t$ into $t'$ or not (together with the removal of the outgoing edge from $r'$ in $t$).
	The trees $(t,r)$ which may lead to $(t',r')$ do not have the (oriented) edge $(r,r')$ but an edge from $r'$ to one of the neighbors $u$ of $r'$ in $G$; this is the only difference between $t$ and $t'$.
		The trees $(t,r)$ having $(t',r')$ as neighbour (see Fig . \ref{fig:NN}) are then parameterized by these neighbours $u$ of $r'$, and the balance equation (if one assumes $P(t,r)=\prod_{e \in E(t,r)}\ra{M}_e$), 
		\[\sum_{\stackrel{(t,r): (t',r')}{\textrm{ is a neighbour}}} P((t,r))Q((t,r),(t',r'))= \l(\sum_{u \textrm{ neighbour of  }r'} \ra{M}_{r',u} \frac{\l[\prod_{e \in E(t',r')} \ra{M}_e\r]}{\ra{M}_{r,r'}} \r) \ra{M}_{r,r'}=\prod_{e \in E(t',r')} \ra{M}_e \]
		is satisfied.\\ As mentioned by Aldous \cite{Al90}, the irreducibility follows the fact that the $\LET$ of the sequence of steps that forms the depth first traversal of a tree is the tree itself.}
\end{proof}
\begin{figure}[htbp]
  \centerline{\includegraphics[scale=0.8]{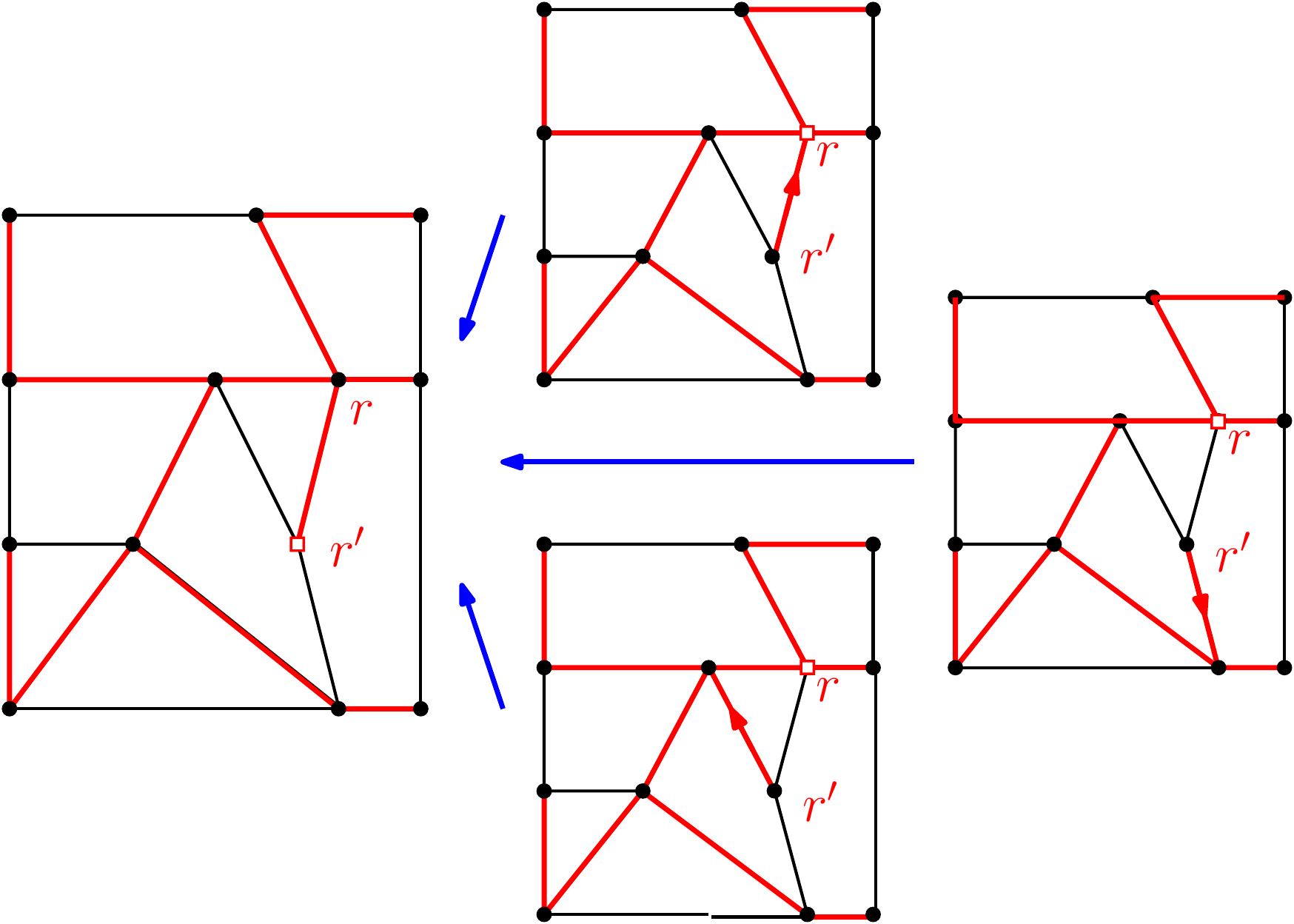}}
  \caption{\label{fig:NN} The antecedent of a given tree for the tree Markov chain.}
\end{figure}

\begin{rem} In Hu \& al. \cite{HLT}, it is established that the first entrance tree and the last exit tree taken at the cover time of the Markov chain, have same distribution.
    \end{rem}

\subsection*{Acknowledgments}
    We	thank the anonymous	referee for her/his corrections and comments.
\small

\bibliographystyle{abbrv}

% \bibliography{mybib}

\begin{thebibliography}{}

\end{thebibliography}


\begin{thebibliography}{10}

\bibitem{Al90}
D.~J. Aldous.
\newblock The random walk construction of uniform spanning trees and uniform
  labelled trees.
\newblock {\em SIAM Journal on Discrete Mathematics}, 3(4):450--465, 1990.

\bibitem{Bro89}
A.~Z. Broder.
\newblock Generating random spanning trees.
\newblock In {\em FOCS, vol. 89}, pages 442--447, 1989.

\bibitem{CF}
P.~Cartier and D.~Foata.
\newblock {\em Problèmes combinatoires de commutation et réarrangements}.
\newblock Lecture notes in mathematics, 1969.

\bibitem{DF}
P.~Diaconis and W.~Fulton.
\newblock A growth model, a game, an algebra, {L}agrange inversion, and
  characteristic classes.
\newblock volume~49, pages 95--119 (1993). 1991.
\newblock Commutative algebra and algebraic geometry, II (Italian) (Turin,
  1990).

\bibitem{HLT}
Y.~Hu, R.~Lyons, and P.~Tang.
\newblock {A reverse {A}ldous–{B}roder algorithm}.
\newblock {\em Annales de l'Institut Henri Poincaré, Probabilités et
  Statistiques}, 57(2):890 -- 900, 2021.

\bibitem{Jarai09}
A.~A. Járai.
\newblock The uniform spanning tree and related models, 2009.
\newblock
  \url{http://www.maths.bath.ac.uk/\%7Eaj276/teaching/USF/USFnotes.pdf}.

\bibitem{CK}
C.~Krattenthaler.
\newblock The theory of heaps and the cartier–foata monoid.
\newblock In {\em Appendix of the electronic edition of Problèmes
  combinatoires de commutation et réarrangements}. 2006.

\bibitem{PW98}
J.~G. Propp and D.~B. Wilson.
\newblock How to get a perfectly random sample from a generic markov chain and
  generate a random spanning tree of a directed graph.
\newblock {\em J Algorithms}, 27(2):170--217, 1998.

\bibitem{VX}
G.~X. Viennot.
\newblock Heaps of pieces, i : Basic definitions and combinatorial lemmas.
\newblock In G.~Labelle and P.~Leroux, editors, {\em Combinatoire
  {\'e}num{\'e}rative}, pages 321--350, Berlin, Heidelberg, 1986. Springer
  Berlin Heidelberg.

\bibitem{Wil96}
D.~B. Wilson.
\newblock Generating random spanning trees more quickly than the cover time.
\newblock In {\em Proceedings of the twenty-eighth annual ACM symposium on
  Theory of computing}, pages 296--303, 1996.

\bibitem{DZ}
D.~Zeilberger.
\newblock A combinatorial approach to matrix algebra.
\newblock {\em Discrete Mathematics}, 56(1):61 -- 72, 1985.

\end{thebibliography}

\end{document}